\documentclass[11pt]{article}

\usepackage{xcolor}

\usepackage{enumerate}

\usepackage{hyperref}

\usepackage{amsmath}
\usepackage{amsthm}
\usepackage{amssymb}
\usepackage{amsfonts}
\usepackage{amsrefs}
\usepackage{bbm}
\usepackage[normalem]{ulem}

\usepackage[title, toc]{appendix}

\usepackage{tikz}
\usetikzlibrary{arrows}
\usetikzlibrary{decorations.markings}
\usepackage{color}
\usepackage{graphicx}
\usepackage{fullpage}
\usepackage{float}
\usepackage{wrapfig}
\usepackage{caption}
\usepackage{subcaption}
\usepackage{listings}
\usepackage{verbatim}
\numberwithin{equation}{section}

\DeclareMathOperator{\Var}{Var}
\DeclareMathOperator*{\diag}{diag}

\newcommand{\beq}{ \begin{equation} }
\newcommand{\eeq}{ \end{equation} }
\newcommand{\beqq}{ \begin{equation*} }
\newcommand{\eeqq}{ \end{equation*} }

\newcommand{\ip}[1]{\langle {#1} \rangle }

\def \Z {\mathbb{Z}}
\def \dd {\mathrm{d}}

\def \a {\alpha}
\def \b {\beta}
\def \g {\gamma}

\def \d {\delta}
\def \la {\lambda}
\def \s {\sigma}

\def \r {\rho}

\newcommand{\bR}{\mathbb{R}}
\newcommand{\bN}{\mathbb{N}}

\newcommand{\bE}{\mathbb{E}}
\newcommand{\bP}{\mathbb{P}}

\newcommand{\cF}{\mathcal{F}}

\def \cF {\mathcal{F}}

\def \RR {\mathbb{R}}

\def \cN {\mathcal{N}}
\def \cG {\mathcal{G}}

\newcommand{\ii}{\mathrm{i}}

\newtheorem{theorem}{Theorem}[section]

\newtheorem{lemma}[theorem]{Lemma}
\newtheorem{conjecture}[theorem]{Conjecture}

\newtheorem{definition}[theorem]{Definition}

\theoremstyle{remark}
\newtheorem{remark}[theorem]{Remark}

\DeclareMathOperator{\TW}{TW}

\begin{document}

\title{Free energy fluctuations in SK and related spin glass models: \\ A literature survey}

\author{Elizabeth W. Collins-Woodfin\footnote{Department of Mathematics, University of Oregon,
Eugene, OR, 97403, USA \newline email: \texttt{collinsw@uoregon.edu}} 
\and 
Han Gia Le\footnote{Department of Statistics and Actuarial Science, University of Waterloo, Waterloo, ON, N2L 3G1, Canada \newline email: \texttt{han.le1@uwaterloo.ca}}}

	\date{\today}

	\maketitle

\begin{abstract}
Over the past 50 years, spin glass models have generated a broad range of literature in mathematics, physics, and computer science.  There has been much progress in characterizing and proving the limiting free energy of various models, stemming from the original formulas of Parisi.  Comparatively less is known about the more detailed topic of free energy fluctuations.  
This paper concerns a family of models in which there has been considerable progress on fluctuations, namely the Sherrington-Kirkpatrick (SK) and spherical Sherrington-Kirkpatrick (SSK) models, along with their multi-species analogs.  We present a survey of the literature on free energy fluctuations in these 2-spin models, discussing results from different temperature regimes, with and without an external field, including results on phase transitions.
\end{abstract}	

\section{Introduction}

The Sherrington-Kirkpatrick model \cite{SherringtonKirkpatrick1975} is one of the earliest and most well-studied models in the spin glass literature (see \cite{sherringtonkirkpatrick2025} for a brief history).  While it was initially motivated by physics, it is also interesting from the perspective of optimization and, despite being simple relative to other models, it continues to be a source of open problems.
The SK model and its variants fall within the class of \textit{mean-field}, or infinite-range models, meaning that the interactions between pairs of spin sites have a distribution that is independent of any notion of distance between the sites.  While these models do not mimic the physical reality of particle interactions that decay with distance, they nevertheless exhibit rich mathematical phenomena, as we will see.

\paragraph{Scope of the survey}
The spin glass literature is immense, so the focus of this survey is necessarily narrow, guided by the following question:
\textit{What is known about free energy fluctuations in 2-spin, mean-field models?}
More specifically:
\begin{itemize}
\item The models we discuss \textit{do include} SK and spherical SK, with and without an external field, or with a Curie-Weiss term, but \textit{do not include} $p$-spin or mixed-spin models.
\item Under the umbrella of ``mean-field", we \textit{do include} multi-species models but \textit{do not include} sparse models or models with geometric/graphical structure to their interactions.
\item Within the free energy literature, we focus on \textit{fluctuation results} and only briefly mention seminal works related to the limiting free energy, Parisi formulas, and AT-line.  We briefly discuss overlaps as they relate to free energy, but do not emphasize the literature on overlaps.
\item We focus on the case where the interaction matrix has Gaussian entries.
\item For literature beyond the scope of this survey, we provide a brief list of references in Section \ref{sec:beyondscope}.
\end{itemize}
We have chosen this scope in order to highlight some exciting themes that emerge within the literature on free energy fluctuations of 2-spin, mean field models.  In particular, the discrete and spherical models exhibit strikingly similar behavior at high temperature, but this similarity breaks down as one moves through the critical temperature threshold and into the low temperature regime.  Fluctuations in the SSK model are driven by deep connections to random matrix theory, which persist at all temperatures and provide insight into various phase transitions within this model. In contrast, SK model exhibits a complex energy landscape, in which random matrix methods do not suffice but one needs other techniques such as Gaussian interpolation and cavity methods.

We hope this survey provides an instructive summary of what is known (and what remains open) regarding free energy fluctuations in SK, SSK, and related models.  We look forward to seeing how other scholars will extend and build upon these results!

\subsection{Model set-up}
\paragraph{Sherrington--Kirkpatrick (SK) model} This model in statistical mechanics was introduced in 1975 to study magnetic alloys \cite{SherringtonKirkpatrick1975}.  It involves a spin vector $\s=(\s_1,...,\s_N)^T$ in the configuration space $\{-1,1\}^N$.  For any $\s$, we define the Hamiltonian
\beqq
H_N(\s)=\frac1{2\sqrt N}\sum_{i,j=1}^N g_{ij}\s_i\s_j,
\eeqq
where $(g_{ij})_{i,j=1}^N$ are independent standard Gaussian random variables. Equivalently, one can write $H_N(\s)=\frac12\s^TM\s$ where $M$ is a random matrix from the Gaussian Orthogonal Ensemble (GOE)\footnote{A GOE matrix $M$ is symmetric with independent entries (up to symmetry) where $M_{ij} \sim \cN(0,N^{-1}(1+\delta_{ij})).$}.  Using this Hamiltonian, one can define the \textit{Gibbs measure} on the configuration space, 
\[G(\s)=\frac{\exp(\b H_N(\s))}{Z_N}\qquad \text{where }Z_N:=\sum_{\s \in \{-1,1\}^N} \exp\left(\b H_N(\s)\right). \]
The \textit{free energy} of the model at each inverse temperature $\b>0$ is defined as 
\beqq
F_N(\b) = \frac1N\log Z_N.
\eeqq
As $N\to\infty$, the free energy $F_N(\b)$ converges almost surely to some $F(\b)$.  This limiting free energy is known for every $\b>0$, established in breakthrough works \cite{Parisi1979, Parisi1980} and proven rigorously in \cite{Guerra03, TalagrandSK, Panchenk14}.  The fluctuations are also known in the high temperature regime $\b<1$, but remain open for the low temperature regime.  The fluctuation results will be discussed in Section \ref{sec:SK}. For the fluctuations of this model and other models of this survey, we focus particularly on the large-$N$ asymptotics.

\paragraph{Spherical Sherrington--Kirkpatrick (SSK) model} The SK model has a continuous analog, SSK, in which the configuration space is the hypersphere of radius $\sqrt{N}$ in $\RR^N$, i.e.
\[
S_{N-1}:=\left\{\s\in\RR^N:\|\s\|=\sqrt{N}\right\}
\]
This is a natural extension of SK to continuous space, since $S_{N-1}$  embeds the SK configuration space $\{\pm1\}^N$.  The Hamiltonian and free energy for SSK take the same definitions as for SK except that the partition function $Z_N$ is now an integral rather than a finite sum; i.e.,
\[
Z_N:=\int_{S_{N-1}}\exp\left(\b H_N(\s)\right)\dd\omega_N(\s),
\]
where $\omega_N$ is the uniform measure on $S_{N-1}$.  The SSK model was first introduced by Kosterlitz, Thouless, and Jones \cite{kosterlitz1976spherical}. They provide a non-rigorous computation of its limiting free energy, which is extended to a broader class of models in \cite{crisanti1992sphericalp} and rigorously proven in \cite{TalagrandSSK} and \cite{Chen13}.  The fluctuations have also been proved in the high and low temperature regime as well as the critical temperature threshold (see Section \ref{sec:SSK}).

\paragraph{Ground state energy} Closely related to the free energy is the ground state energy, which is defined as 
\beq
\max_{\s\in\Sigma_N}\frac1NH_N(\s)\quad\text{or equivalently }\quad\lim_{\b\to\infty}\frac1\b F_N(\b)
\eeq
where $\Sigma_N$ denotes $\{\pm1\}^N$ (resp.\@ $S_{N-1}$) in the case of SK (resp.\@ SSK) and $F_N(\b)$ is the free energy of the corresponding model.  Because of these two equivalent definitions, the ground state energy can be viewed as the free energy of the model at zero temperature and the free energy can be viewed as a positive temperature relaxation of a constrained optimization problem. 

\paragraph{SK (or SSK) with an external field} 
Both the SK and SSK model can include an external magnetic field.  In this type of model the Hamiltonian becomes
\beq
H_N(\s)=\frac{1}{2\sqrt{N}}\sum_{i,j=1}^Ng_{ij}\s_i\s_j+h\sum_{i=1}^Nv_i\s_i
\eeq
where $v=(v_1,...,v_N)$ is called the external field vector and can be chosen to be either random or deterministic.  The properties of the model are similar with either choice, provided the vector is independent from the disorder variables $\{g_{ij}\}$.

\paragraph{SK (or SSK) with ferromagnetic/Curie-Weiss interaction} 
Another variation of the standard SK or SSK model is the model with ferromagnetic Curie-Weiss interaction.  Instead of the linear external field term discussed above, the Curie-Weiss term is quadratic in the spin variables.  More specifically, 
\beq\label{eq:CW}
H_N(\sigma)=\frac1{2\sqrt{N}}\sum_{i,j=1}^Ng_{ij}\sigma_i\sigma_j+\frac{J}{2N}\sum_{i,j=1}^N\sigma_i\sigma_j
\eeq
where $J$ is called the \textit{coupling constant}.  This model exhibits three phases depending on the temperature and coupling constant (see Figure \ref{fig:phasediagram}).  The phases are: spin glass, paramagnetic, ferromagnetic, which correspond to the cases in which $\max\{1,\b^{-1},J\}$ is equal to $1,\b^{-1},J$ respectively.  Note that the previously mentioned high- and low-temperature regimes of the standard SK or SSK model (with $J=0$) fall within the paramagnetic and spin glass phases respectively.
\begin{figure}[H]
\centering
\includegraphics[width=.3\textwidth]{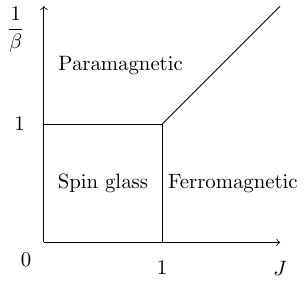}
\caption{Phase diagram for SSK with ferromagnetic Curie-Weiss interaction, adapted from \cite{BLW18}.}
\label{fig:phasediagram}
\end{figure}

\paragraph{Multi-species models} The definitions and discussion of the literature for these models are deferred to Section \ref{sec:multispecies}.

\paragraph{Conventions}
With respect to definitions and notations, there are a variety of conventions that are not entirely consistent within the spin glass literature.  We summarize here the conventions that we choose in this paper.  When we cite theorems from other papers, we translate their results to match our conventions.  In most cases, the choice of convention is simply a rescaling and does not substantively change the results. We highlight any cases in which the choice of convention matters.
\begin{itemize}
\item \textit{Scaling of inverse temperature} -- We define the Hamiltonian as $H_N(\s)=\frac1{2\sqrt N}\sum_{i,j=1}^N g_{ij}\s_i\s_j$.  In contrast, some papers omit the factor $\frac12$, using the definition $H_N(\s)=\frac1{\sqrt N}\sum_{i,j=1}^N g_{ij}\s_i\s_j$.  This amounts of a change of variable for $\b$ in the formula of $Z_N$, such that the critical inverse temperature is $\b=1$ with our definition but $\b=\frac12$ with the other definition.
\item \textit{Scaling of free energy (with respect to $N$)} -- We define the free energy as $F_N(\b)=\frac1N\log Z_N$.  In contrast, some papers omit the factor $\frac1N$, using the definition $F_N(\b)=\log Z_N$.  In our scaling, the free energy has order 1 and the fluctuations of the free energy have order between $N^{-1}$ and $N^{-1/2}$, depending on the model regime (variance order between $N^{-2}$ and $N^{-1}$).  In literature using the other convention, the free energy is of order $N$ and the fluctuations are scaled accordingly.
\item \textit{Scaling of free energy (with respect to $\b$)} -- 
In contrast to our definition $F_N(\b)=\frac1N\log Z_N$, some papers rescale the free energy by $\b$, using the definition $F_N(\b)=\frac{1}{\b N}\log Z_N$.  That definition can be useful when one is interested in the asymptotics of the free energy as $\b$ goes to 0 or $\infty$. In the finite, positive $\b$ regime that we consider, it is simply a rescaling.
\item \textit{Scaling of ground state energy} -- We define the ground state energy as $\max_{\s\in\Sigma_N}\frac1NH_N(\s)$.  In contrast, papers that omit the factor $\frac1N$ in the definition of free energy also omit it in the definition of ground state energy.
\item \textit{Diagonal of the disorder matrix} -- We take the disorder variables $\{g_{ij}\}$ in the Hamiltonian to be standard normal random variables for all pairs $(i,j)$.  Some papers take the diagonal entries $\{g_{ii}\}$ to be zero.  This choice is motivated by the physics context in which it does not make sense to define the interaction of a particle with itself.  The effect of these entries is negligible in the asymptotics, so the results are equivalent for either convention.  However, setting all $g_{ij}$ to be standard Gaussian random variables allows one to directly apply theorems about GOE from the random matrix literature.
\item \textit{External field vector} -- As mentioned above, the external field vector can be take to be random (e.g. standard Gaussian) or deterministic (e.g. all-one vector) or even a sum of deterministic and random vectors.  We do not fix a convention here, but specify it in each theorem.
\end{itemize}

\subsection{Some foundational results}

Before discussing fluctuation results for the free energy of SK and SSK models, we briefly touch on the rich history of their limiting free energies. These two models are said to have \textit{self-averaging} free energy, meaning that $F_N$ stays bounded (below and above) and the fluctuations diminish as $N \to \infty$. In many cases, convergence of the disorder expectation $\bE F_N$ is demonstrated by providing an explicit description of its limit. For Ising spins, the celebrated Parisi formula, proposed by Parisi \cite{Parisi1979, Parisi1980}, states that the limiting free energy is the infimum of a functional, over all probability measures supported on $[0,1]$. The formula is stated for general \textit{mixed $p$-spin models} (see Section \ref{sec:beyondscope}), and mathematical breakthroughs proving the formula are due to Talagrand \cite{TalagrandSK} and Panchenko \cite{Panchenk14}. In the special case of the SK model, the functional has an explicit expression that involves a solution of a PDE, to be solved backward in time with space-time parameter in $\bR \times [0,1]$, and initial condition at $t=1$. An analogous formula for the spherical model was developed by Crisanti and Sommers \cite{crisanti1992sphericalp}, and mathematical proofs of the Parisi formula in this case are due to Talagrand \cite{TalagrandSSK} and Chen \cite{Chen13}. In the special case of the SSK model, we have a closed-form expression (see Eq.\@ \eqref{eq:limitingF_SSK}).

After the development of the Parisi formula, a central question to the study of spin glasses concerns minimizers of the Parisi functional. Existence and uniqueness of the minimizer, known as the \textit{Parisi measure}, is highly non-trivial, especially for discrete spins \cite{AC15_unique}. Another important question concerns the temperature-external field phase transition of the Parisi measure $\mu$, leading to the topic of replica symmetric versus replica symmetric breaking regimes. The model is said to be in the \textit{replica symmetric} (RS) phase if $\mu$ is a singleton, and is said to have \textit{replica symmetric breaking} otherwise. Within RSB, we have $k$-step Replica Symmetry Breaking (kRSB) if $\mu$ has $k+1$ atoms and full Replica Symmetry Breaking (FRSB) if $\mu$ has an absolutely continuous part. 

Under zero external field, the SK model is shown to be RS at high temperature $0<\b <1$ by \cite{AizenmanLebowitzRuelle} and is not RS at low temperature $\b>1$ by Toninelli \cite{Toninelli_2002}. A recent result of Zhou \cite{Zhou25} shows that at temperatures slightly below the critical value, i.e. $1 < \b <1 +\eta$ for some $\eta>0$, the system is in the FRSB phase. The picture gets more complicated under positive external field for the SK model. In this case, \cite{Toninelli_2002} shows that the system exhibits RSB when $(\b, h)$ is such that
\[
\b^2 \bE \left[\frac1{\cosh^4(\b z\sqrt{q} +h)}\right]>1,
\]
where the expectation is with respect to the standard Gaussian random variable $z$ and $q$ is the solution to $q = \bE \tanh^2(\b z\sqrt{q} +h)$. The equality case $
\b^2 \bE \left[\frac1{\cosh^4(\b z\sqrt{q} +h)}\right]=1$ is known as the AT line, named after de Almeida and Thouless \cite{dAT_1978}. For more discussions on the AT line and RSB regime, see \cites{Toninelli_2002,JT_2017,BY_2022, Bolthausen14, Bolthausen19} (single-species) and \cites{BatesSlomanSohn19, DeyWu21, Kim25} (multi-species) for the discrete case and \cite{JT_2018,BatesSohn22b}, for example, for the spherical case. There has also been work studying the geometry of the Gibbs measure, including the shattered phase, which is a subset of the RS phase \cite{Subag_2017, BAJ_2024, GJK_2025, ElAMS_2025, crisanti1993spherical}. In addition, a recent work \cite{Sellke_marginalstability} considers spherical spin glass models at low temperature and draws a connection between the near-global energy maxima and the full RSB property at zero temperature.

The Parisi measure is closely connected to the \textit{order parameter}, and studies of the RS and RSB regimes have also been done in terms of the order parameter (e.g., \cites{BatesSlomanSohn19}). In mean-field models such as SK and SSK, the order parameter refers to the limiting distribution of the absolute value $|R_{1,2}|$, where $R_{1,2}$ denotes the replica overlap
\[
R_{1,2} := \frac1N \sum_{i=1}^N \s^1_i\s^2_i
\]
between two independent samples $\s^1, \s^2$ from the Gibbs measure. For the SSK model, the order parameter is a single-point mass at all temperatures, regardless of external field strength. Thus, we say the model is RS in the whole parameter space. 

It is expected that the Parisi measure and the order parameter coincide, yet at the moment, this is verified only for a certain class of mixed $p$-spin models (which does not include the 2-spin or any pure $p$-spin SK model). For further discussion on the two quantities, see \cite{AC15, AC15_unique, ACZ_2020}.

\subsection{Brief remarks on literature beyond this survey}\label{sec:beyondscope}

Much of the recent work in mean-field spin glasses has focused, not on the SK or SSK model, but on their $p$-spin and mixed spin generalizations.  The Hamiltonian for a pure $p$-spin model is
\[H_{N,p(\s)}=N^{-\frac{p-1}{2}}\sum_{i_1,...,i_p=1}^Ng_{i_1,...,i_p}\s_{i_1}\cdots\s_{i_p}\]
where SK and SSK are pure 2-spin models.  A mixed-spin model has Hamiltonian $H(\s)=\sum_{p=1}^\infty \b_pH_{N,p}(\s)$ for coefficients $\{\b_p\}$ with a suitable decay rate (a constraint typically imposed is $\sum_{p=1}^\infty2^p\b_p^2< \infty$).

In this more generalized setting, the limiting free energy does not have a closed-form formula such as for SSK and SK (at high temperature). Rather, the Parisi formula (or the Crisanti-Sommers formulation for the spherical case) provides a variational representation of $F(\b)$. In addition, formulating the limiting free energy in the general multi-species setting is an area of continuing research. A Parisi formula for the multi-species SK model is obtained in \cites{Barra15, Panchenko15}, and the Crisanti-Sommers-type expression for (mixed $p$-spin) multi-species SSK model is derived in \cite{BatesSohn22}. Both of these results rely on the assumption of convexity of the \textit{mixture polynomial} for the proof of the upper bound. In addition, \cite{Subag23} computes the limiting free energy for the (pure $p$-spin) multi-species spherical SK model via the so-called TAP approach. The statement relies on the assumption of convergence of certain free energies, instead of the convexity assumption. In the case of ground state energy, \cite[Theorem 1.5]{HuangSellke_topologicaltrivialization} implies a closed-form formula for the multi-species spherical SK model at all external field strengths with no convexity assumptions (see, e.g.\@, \cite{SubagZeitouni_2017} for an example of deriving the ground state energy from such a complexity result). 

While this survey focuses on the free energy fluctuation results in the SK and SSK models ($p=2$), there has also been progress on fluctuations for $p>2$ in the RS phase. See, e.g., \cite{Bovier_rem, Bovier_twoscales, JLM_tensorPCA}.  The fluctuations of the ground state energy for general spherical models have also been characterized using the replica method from physics \cite{Fyodorov_groundstateLDP}. In addition to free energy fluctuation results, there has also been work on the fluctuations of overlaps.  See, e.g., \cite{Baiketal21,CW21,landon2020fluctuations,LS22,nguyen2018central}. The fluctuation results in these works are with respect to the Gibbs measure and hold for fixed disorder. They are all related to SSK and are proven using a contour integral technique. Treatment of overlaps of SK is done in \cite[Section 1.10]{Talagrandvol1}.

Another important line of research is related to complexity of the energy landscape, which refers to the number of critical points (of various orders) of the Hamiltonian.  Much of the progress in this area stems from the seminal work of Auffinger, Ben Arous, and \v Cern\'y in the spherical setting \cite{Auffinger13,AuffingerBenArousCerny}.  See \cite{BenArousSubagZeitouni,AuffingerChen_energylandscape,Beliusetal_triviality,FyodorovLeDoussal2018,Subag_complexityspherical, HuangSellke_topologicaltrivialization,Kivimae_complexity, JT_2017lowtemp, Sellke_marginalstability,Fyodorov_complexity} and the references therein for recent progress in this area.

The limiting free energy and fluctuations of the free energy are \textit{static questions} about spin glass models, but one can also consider \textit{dynamical questions} related to how a spin distribution evolves under a particular algorithm (e.g., Glauber dynamics for SK or Langevin dynamics for spherical models).  Key ideas that are studied in relation to dynamics include mixing times \cite{BAJ_2018spectralgap,BAJ_2024,EldanKoehlerZeitouni_mixing, BB_LSI, GheissariJagannath_sphericaldynamics, AnariKoehlerVuong_localization} and aging \cite{BenArousGun, Gayrard_rem, CW_rem}. Some foundational works on spin glass dynamics include \cite{CugliandoloKurchan93, crisanti1993spherical, Grunwald96, BADG_CKdynamics}. For a small sampling of additional works on these topics, see \cite{Dandietal_dynamics, Sellke_langevin,BAGJ_sphericaldynamics,HuangSellke_optimizationspherical} and the references therein. 

Although the results discussed in this survey are presented in the case of Gaussian interaction matrix, many of them are extended to a broader class of matrices (e.g.,\@ Wigner matrices). However, there are also works examining other types of distributions, such as heavy-tailed or deformed matrices \cite{ChenKimSen25, KimLee24, LeeLi23}. 

Finally, there is wealth of literature beyond the context of mean-field spin glasses, related to spin systems with nearest-neighbor or graph structure.   Other models beyond the scope of this survey include diluted spin models, vector-valued spins, Potts model, and many more.  Such models have been studied not only in the mathematical literature, but also in physics and computer science.  Our literature review is inevitably incomplete due the limited scope of our paper and the vast amount of literature in this area.  We apologize in advance to anyone whose work we have omitted and welcome the reader to reach out with additional references that may be of interest.

\subsection{Organization of the paper}
The remainder of this paper is split into three sections with Section \ref{sec:SK} focusing on the SK model, Section \ref{sec:SSK} on the SSK model, and Section \ref{sec:multispecies} on multi-species models (both discrete and spherical).

\section{Free energy fluctuations in the SK model}\label{sec:SK}

In this section, we discuss fluctuation results for the free energy of the SK model, where spins take value $+1$ or $-1$. As a prototypical model of spin glasses, the SK model has been the subject of an extensive amount of work.

We first consider the model with no external field, before presenting the literature for the model under the presence of one. In some cases, the situation actually becomes simpler when there is an external field. The external field term acts as a bias in one direction, breaking the symmetry among the spin configurations, which provides an analytical advantage.

\subsection{Results for SK model without external field}\label{subsec:sknoext}

\paragraph{High temperature results.} After the SK model was introduced in 1975,  Aizenman, Lebowitz and Ruelle showed in 1987 that, in the high-temperature regime, the free energy fluctuations are asymptotically Gaussian \cite{AizenmanLebowitzRuelle}.
The results were shown for i.i.d.\@ interaction coefficients $(J_{ij})_{1 \leq i<j \leq N}$ where $J_{12}$ is symmetric about zero, has moments of all orders, and satisfies  $\bE \exp(\a J_{12})<\infty$ for some $\a>0$. In the Gaussian set-up of this survey, the result is as follows.
\begin{theorem}[SK at high temperature \protect{\cite[Proposition 2.2]{AizenmanLebowitzRuelle}}]\label{thm:ALR}
    For all $\b<1$, we have the convergence in distribution as $N \to \infty$
    \[
    N\left(F_N(\b) - \left(\log 2 + \b^2/4\right)\right) \to X, 
    \]
    where $X \sim \cN(-\frac12 \s^2, \s^2)$ and $\s^2=-\frac12\log(1-\b^2)-\frac{\b^2}{2}$.
\end{theorem}
\begin{remark}
Theorem \ref{thm:ALR} is stated using the convention that the diagonal entries of the interaction matrix are zero. If the interaction matrix is a GOE matrix, the expression of $\s^2$ becomes $\s^2 = -\frac12\log(1-\b^2)$. For more details, see \cite{AKJ_2020} for the Ising spins and \cite{BL17} for the spherical case.
\end{remark}

The main result in \cite{AizenmanLebowitzRuelle} is stated as the convergence of $\Z_N/\bE Z_N$ to a log-normal random variable $T$, i.e., $T = \exp(u - \frac12 \bE u^2)$ for a centered Gaussian random variable $u$. The moment method is not applied directly to this ratio, as log-normal random variables are not uniquely determined by their moments. Instead, the proof is obtained via cluster expansion, which works for general symmetric disorder. In this approach the configuration space is decomposed into subgraphs of an underlying graph. First, the partition function is written as
\[
Z_N = 2^N \left(\prod_{i<j} \cosh(\b J_{ij}/ \sqrt{N}) \right) \hat{Z}_N, \quad 
\hat{Z}_N = \sum_{\Gamma} \prod_{b \in E(\Gamma)} \tanh(\b J_b/\sqrt N),
\]
where the sum is taken over all graphs $\Gamma$ of a specified form, and the product is taken over all edges of $\Gamma$. The next step is to show that the main contribution to $\hat{Z}_N$ comes from clusters of cycles, and the fluctuations are asymptotically Gaussian via moment method. The high-temperature condition $\b<1$ guarantees that the contribution from large graphs decays exponentially in graph size.

Other methods have also been used to prove the Gaussian fluctuations at high temperature, such as stochastic analysis \cite{FrohlochZegarlinski, CometsNeveu} and the moment method \cite{Talagrandvol2}. In \cite[Chapter 11]{Talagrandvol2}, Talagrand refines the above Gaussian fluctuation result, providing an upper bound on the second-order term in the expansion of the moments of the quantity $N(F_N(\b) - F(\b)) - \frac14 \log(1-\b^2)$. Using one-variable calculus and Gaussian integration by parts, the author derives a recurrence on moments of $N(F_N(\b) - F(\b))$, which allows for recovering the moments via simple induction and integration. 

In \cite{CometsNeveu}, Comets and Neveu use a stochastic analysis approach, where they consider the process $\{\tilde{Z}_N(t)\}_{t \in [0,1]}$ given by
\[
\tilde{Z}_N(t):= \frac1{2^N}\sum_\s e_N(t, \s),
\]
where
\[
e_N(t,\s):= \exp \left[\frac1{\sqrt N}\sum_{i<j}B_{ij}(t)\s_i\s_j - \frac{(N-1)t}{4}\right].
\]
The sum on the right hand side resembles the Hamiltonian of the SK model, with the Brownian motions at time $t$ replacing the Gaussian interactions. Indeed, this process $\tilde{Z}_N(t)$, at $t =\b^2$, is related to the partition function $Z_N(\b)$ of the SK model via
\[
\tilde{Z}_N(\b^2) \stackrel{d}{=} Z(\b)/\bE Z(\b).
\]
On the other hand, $\tilde{Z}_N(t)$ is the exponential martingale: 
\beq\label{eq:expM}
\tilde{Z}_N(t) = \exp \left[M_N(t)-\frac12\ip{M_N}(t)\right],
\eeq
where $\{M_N(t)\}_{t \in [0,1]}$ is a martingale given by
\[
M_N(t)= \int_0^t \dd \tilde{Z}_N(s)/\tilde{Z}_N(s)
\]
and $\ip{M_N}$ is its bracket (not to be confused with Gibbs average). The bracket is the unique random process such that $M_N^2(t)-\ip{M_N}(t)$ is a martingale, and is equivalent to the quadratic variation when the local martingale is continuous. The next step is to show that $M_N$ is a centered Gaussian process via $\ip{M_N}$. By Rebolledo's CLT for local martingales \cite{Rebolledo_1980}, if $M_N$ is a martingale whose bracket $\ip{M_N}$ converges in probability to a deterministic function $\phi$ defined on $[0,1)$, then the martingale $M_N$ itself converges in probability to a centered Gaussian process on $[0,1)$ with independent increments and variance $\phi$. Gaussian fluctuations of the free energy then follows directly from \eqref{eq:expM}. Indeed, $\ip{M_N}(t)$ converges in probability to the deterministic function $\phi(t)= \frac12 \left(\log \frac1{1-t} - t\right)$. The expression of $\phi$ matches with the variance established in Theorem \ref{thm:ALR}.

\paragraph{Low temperature results.} In contrast to the high-temperature regime, the fluctuations of the free energy in the low-temperature regime remains an open problem. Regardless, we discuss the upper bound and lower bound that are obtained for the fluctuations of the free energy and those of the ground state energy (the zero-temperature case). In our discussion, we follow the convention that ``finite temperature" refers to finite, non-zero temperature.

First, we clarify some terminologies.  Given a sequence $\{X_n\}_{n\geq 1}$ of random variables, we say that $X_n$  has fluctuation exponent $\r$ if it has fluctuations of order $n^{\r+o(1)}$. Moreover, we say $X_n$ has a fluctuations of order at most $\Delta_n$ if there is a constant $c>0$ such that for all large $n$, we have $\Var(X_n) \leq c \Delta_n^2$. In the case of a lower bound, we follow \protect{\cite[Definition 1.1]{Chatterjee19}} and say that $X_n$ has a fluctuations of order at least $\delta_n$ if there are constants $c_1,c_2>0$ such that for all large $n$, for all $-\infty<a\leq b<\infty$ where $b-a\leq c_1 \delta_n$, we have $\bP(a \leq X_n \leq b) \leq 1-c_2.$

In the zero-temperature regime, there has been an extensive effort on determining the fluctuation exponent $\r$ of the ground state energy. Several works, including \cites{bouchaud2003energy, Boettcher_2005}, support the claim that $\r = -3/4$, while other works such as \cites{crisanti1992replica, conjSK1, ParisiRizzo09} predict that the exponent is actually $-5/6$. Notably, it is incredibly hard to obtain numerical estimation that distinguishes between the two scalings, and exact algorithms are limited to very small systems. However, we have that $\r \leq -1/2$ (see \cite[Section 2.2]{Chatterjee19}). A formal argument for $\r \leq -3/4$ via the non-rigorous replica method is provided in \cite{conjSK2}. 

On the mathematical front, Chen, Handschy and Lerman show in \cite{ChenHandschyLerman2018} that the ground state energy fluctuations are $o(N^{-1/2})$. In turn, Chatterjee provides a lower bound in \cite{Chatterjee19}, showing that the ground state energy has fluctuations of at least order $N^{-1}$.

Compared to the zero-temperature regime, it is even more difficult to obtain numerical estimations in the finite temperature regime. It has been proposed in \cite{conjSK2} that this can be done by simulating chaos. However, various works \cites{Chatterjee09, Chatterjee2014book, Chatterjee19} of Chatterjee provide upper and lower bounds on the fluctuations of the free energy. These bounds hold at all finite temperatures. 

\begin{theorem}[Upper bound under finite temperature, \protect{\cite[Theorem 1.5]{Chatterjee09}}]\label{thm:chatterjee_upperbd}
For the SK model, there exists an absolute constant $C$ such that, for any $\b \in (0,\infty)$, 
\[
\Var(F_N(\b)) \leq \frac{C\b^2\log(2+C\b)}{N \log N}.
\]
\end{theorem}
The above upper bound is not sharp for high temperature, as the variance is of order $N^{-2}$ in this regime by \cite{AizenmanLebowitzRuelle}. However, to our knowledge, Theorem \ref{thm:chatterjee_upperbd} is the most recent progress for the low-temperature regime.

In \cite{Chatterjee09}, the theorem is proven using Gaussian interpolation. Although the proof is shown for general mixed $p$-spin SK models, we provide a proof sketch for the 2-spin SK model. We observe that
    \[
    \mathrm{Cov}(H_N(\s), H_N(\s')) = \frac{(\s\cdot \s')^2}{2N} =N\xi(R_{\s,\s'}), 
    \]
    where we set $\xi(x):= x^2/2$ and $R_{\s,\s'} = \frac{\s\cdot \s'}{N}$, which satisfies $|R_{\s,\s'}| \leq 1$. 
Upon this observation, the three key steps of the proof are as follows.
\begin{itemize}
    \item Take two independent copies $H'_N, H''_N$ of $H_N$. Given $t\geq 0$, we define
    \begin{align*}
        H^t_N &:= e^{-t}H_N + \sqrt{1-e^{-2t}}H_N',\\
        H^{-t}_N &:= e^{-t}H_N + \sqrt{1-e^{-2t}}H_N''.
    \end{align*}
    \item Via Gaussian integration by parts, we have the variance identity
    \beq\label{eq:varid_chatterjee}
    \Var(F_N(\b)) =\frac{\b^2}{N} \int_0^\infty e^{-t} \bE \ip{\xi(R_{\s,\s'})}_{t} \dd t,
    \eeq
    where the expectation $\ip{\cdot}_t$ is given by
    \[
    \ip{f}_t = \frac{\sum_{\s, \r}f(\s,\r) e^{\b H^t(\s)+\b H^{-t}(\r)}}{\sum_{\s', \r'} e^{\b H^t(\s')+\b H^{-t}(\r')}}.
    \]
    \item One then shows that there exists an absolute constant $C$ such that, for all $t\geq 0$ and for all $k \in \bN$, 
    \[
    \bE\ip{(\xi(R_{\s,\s'}))^k}_{t} \leq (Ck)^k N^{-k \min \left\{1, t/\left(C\log(1+C\b)\right)\right\}}.
    \]
    Applying this statement with $k=1$ to the variance identity and integrating, we obtain the right hand side of theorem, with the main contribution coming from the integral over the sub-interval $[0, C\log(1+C\b)]$ (contribution from the rest is $O(N^{-2})$). 
\end{itemize}

We would like to highlight the key ingredient that is Gaussian analysis. In all models discussed in this survey (and their mixed $p$-spin extensions), the collection $\{H_N(\s)\}_{\s \in \Sigma_N}$ forms a mean-zero Gaussian process, and consequently, their joint distribution is determined by the covariances. Mean-field models in general do not share this property of Gaussianity (e.g., the Hopfield model \cite{Hopfield} and other models in \cite{Talagrandvol1}). The covariances are functions of the overlap, so controlling the overlap is a key part of the analysis. We will see this recipe of a variance identity via interpolation plus control of the overlap again in the derivation of an upper bound for the variance in the critical temperature case, as done in \cite{ChenLam19}.

A lower bound is also established for all $\b \in (0,\infty)$. In \cite{Chatterjee19}, Chatterjee shows that, at all finite temperatures and under the absence of the external field, the order of fluctuations is at least $N^{-1}$. The idea behind the method, generally speaking, involves perturbing the disorder $\{g_{ij}\}_{i,j=1}^N$ and then comparing the total variation distance between law of original free energy and that of the perturbed one. In the process, convexity in $\b$ of free energy and monotonicity in $\b$ of the limiting free energy are used.

\paragraph{Critical temperature results.}
In light of the above two sections, we observe that the behavior of the free energy in the high-temperature regime is well-understood, while the situation is complicated in the low-temperature regime. Naturally, we question how the free energy fluctuates at the critical temperature, i.e.\@ $\b=1$.

Various works, from numerical computations to heuristic argument, such as \cites{Parisietal_1993, conjSK2}, support the following conjecture.
\begin{conjecture}
For the SK model, at the critical value $\b_c=1$, we expect
\[
\Var(F_N^{SK}(\b_c))= \frac16 N^{-2}\log N + O(N^{-2}), \quad \text{as } N \to \infty.
\]
\end{conjecture}
Notable progress towards this conjecture for the critical temperature regime is \cite{ChenLam19} of Chen and Lam. 

\begin{theorem}[SK model at the critical temperature, \cite{ChenLam19}]\label{thm:chenlam19}
The following statements hold:
    \begin{enumerate}
    \item There exists a constant $C>0$ such that 
    \beq\label{item:chenlam19_crit}
    	\Var\left(F_N(\b_c)\right) \leq C \left( N^{-2}(\log N)^2 +1\right).
    \eeq
    \item For fixed $\a>0$ and $d>0$, there is a constant $C>0$, depending only on $\a$ and $\d$, such that
    \beq
    	\Var\left(F_N(\sqrt{\b_c^2+dN^{-\a}})\right) \leq C \left( N^{-2}(\log N)^2 + N^{-1-\a}\right).
    \eeq
    \end{enumerate}
\end{theorem}

Proof of Theorem \ref{thm:chenlam19} in \cite{ChenLam19} follows the general steps that we highlight for Theorem \ref{thm:chatterjee_upperbd}. In particular, the method requires an upper bound on $\bE \ip{R_{1,2}^2}_t$.

Prior to \cite{ChenLam19}, Talagrand proves in \cite[Section 11.7]{Talagrandvol2} that, for all $\b \leq 1$, we have $\bE \ip{R_{1,2}^2} \leq C/\sqrt{N}$. Next, Chatterjee observes that $\bE \ip{R_{1,2}^2}_{t}$ is a decreasing function of $t$. Combining these two facts with the variance identity \eqref{eq:varid_chatterjee}, one obtains the bound
\[
\Var(F_N(\b_c)) \leq CN^{-3/2}.
\]
By providing a tighter, $t$-dependent bound of the double expectation $\bE \ip{R_{1,2}^2}_{t}$, the authors of \cite{ChenLam19} are able to obtain an improvement \eqref{item:chenlam19_crit}. In order to derive such a bound with explicit dependence on $t$, they employ a quadratic replica coupling, which is introduced in \cite{GuerraToninelli_2002a}. Later in Section \ref{sec:multispecies}, we adopt this method and provide an extension of Theorem \ref{thm:chenlam19} to a subset of \emph{multi-species} models, which contains the SK model.

\subsection{Results for SK model with external field}\label{sec:SK+ext}

We now review the literature on free energy of the SK model under the presence of the external field. In particular, we consider the following Hamiltonian: 
\[
H_N(\s) = \frac1{\sqrt N}\sum_{i,j=1}^N g_{ij}\s_i\s_j  + \sum_{i\leq N} (h+\b_1 g_i)\s_i,
\]
where $g_1, \dots, g_N$ are i.i.d.\@ standard Gaussian random variables. Here, the external field term has both a deterministic direction that is the all-one vector and a random direction that is the gaussian vector $(g_1, \dots, g_N)^T$. The non-negative constants $h$ and $\b_1$ denote the external field strengths. 

In \cite{ChenDeyPanchenko}, Chen, Dey and Panchenko show that, under the presence of an external field, the free energy has Gaussian fluctuations of order $N^{-1/2}$ at all temperatures. Prior to \cite{ChenDeyPanchenko}, fluctuation results for the free energy are verified in the high temperature case in \cite[Theorem 5]{GuerraToninelli_2002b} and in \cite{Tindel_2005}, the latter of which obtains the same CLT statement for a subset of the high-temperature regime via the stochastic calculus approach. 

\begin{theorem}[SK model with external field, \protect{\cite[Theorem 2]{ChenDeyPanchenko}}]\label{thm:chendeypanchenko}
    Suppose $h^2+\b_1^2 \neq 0$. For some $\nu = \nu(h,
    \b_1)>0$, it holds that 
    \[
    \lim_{N\to \infty} d_{\mathrm{TV}}\left(\sqrt{N/\nu}(F_N-\bE F_N), g\right) = 0,
    \]
    where $d_{\mathrm{TV}}$ is the total variation distance and $g$ is a standard Gaussian random variable. 
\end{theorem}

The result holds for mixed $p$-spin model with no odd $p$, and the parameter $\nu$ in the limiting Gaussian distribution is explicitly described. Note that the asymptotics of the expected free energy $\bE F_N$ are not provided in their paper. We direct the readers to \cite{ChenDeyPanchenko} for a full description of $\nu$. 

Combining the theorem with fluctuation results for the case of no external field in Section \ref{subsec:sknoext}, we arrive at the following summary. In the finite temperature case, the free energy fluctuations for the SK model are:
\begin{itemize}
    \item Gaussian and of order $1/N$ at high temp with no external field ($\b<1$ and $h=0$), 
    \item of order at least $1/N$ and at most $1/\sqrt{N \log N}$ at low temperature with no external field ($\b>1$ and $h=0$),
    \item Gaussian and of order $1/\sqrt{N}$ at any temp with external field ($h>0$). 
\end{itemize}

In the zero temperature case, due to the work \cite{ChenHandschyLerman2018} by Chen, Handschy and Lerman, the fluctuations of ground state energy for the SK model are:
\begin{itemize}
    \item $o(N^{-1/2})$ under no external field ($h=0$),
    \item Gaussian and of order $N^{-1/2}$ under positive external field ($h>0$).
\end{itemize}

Furthermore, the transitional regime between zero external field and positive external field is also understood for the high-temperature regime, due to the work \cite{DeyWu23} of Dey and Wu. In particular, by considering the limit $h\downarrow 0$ in the form of
\[
h = \rho N^{-\a} \quad \text{where the constants $\rho, \a$ lie in }(0,\infty),
\]
the authors obtain limit statements for the fluctuations in three regimes: sub-critical ($\a>1/4$), critical ($\a=1/4$), and super-critical ($\a<1/4$). Heuristic ideas on finding the correct threshold $\a=1/4$ are discussed in \cite[Section 1.3]{DeyWu23}, and the details are included in the analysis of path clusters in \cite[Section 5.2]{DeyWu23}. 

\begin{theorem}[SK model at high temp, with weak external field \protect{\cite[Theorem 1.3]{DeyWu23}}]\label{thm:deywu23_SK} 
Consider the SK model with Gaussian disorder. Assume that the external field is scaled as $h = \rho N^{-\a}$ for some positive constants $\rho$ and $\a$. Let
\beq
v_1^2 := -\frac12\log(1-\b^2) - \frac{\b^2}{2} \quad \text{and} \quad v_2^2 := \frac{\b^2}{2(1-\b^2)}.
\eeq
We then have the following statements of convergence in distribution. 
\begin{enumerate}
     \item (Sub-critical) For $\a>1/4$ and $\b<1$,
        \[
        N \left(F_N - \left(\log(2\cosh h)+\frac{N-1}{4N}\b^2\right)\right) \to \cN(-\frac12 v_1^2, v_1^2).
        \]
        \item (Critical) For $\a=1/4$ and $\b<1$,
        \[
        N \left(F_N - \left(\log(2\cosh h)+\frac{N-1}{4N}\b^2\right)\right) \to \cN(-\frac12 \left(v_1^2 + \rho^4 v_2^2 \right), v_1^2+\rho^4 v_2^2).
        \]
        \item (Super-critical) There exists $\b_0\leq 1$ such that, for $\a<1/4$ and $\b<\b_0$,
        \[
        \frac{\sqrt{N}}{h^2}\left(F_N - \left(\bE \log(2\cosh(\b\sqrt{q}\eta+h))+\frac14\b^2(1-q)^2\right)\right) \to \cN(0,v_2^2),
        \]
        where $\eta$ is a standard Gaussian random variable and $q\in [0,1]$ is the solution to 
        \[
        q = \bE \tanh^2(\b \sqrt{q}\eta +h).
        \]
    \end{enumerate}
\end{theorem}

\begin{remark}
Our notations $v_1^2, v_2^2$ differ slightly from those of \cite{DeyWu23}.
\end{remark}

In all three regimes, the fluctuations are asymptotically Gaussian. The comparison between the zero external field and order-one external field is as follows. The three cases---no external field, sub-critical external field, and critical external field---all have the same fluctuation order $N^{-1}$, while the super-critical regime has fluctuation order $N^{-1/2}$, matching the fluctuation order of the order-one external field case in Theorem \ref{thm:chendeypanchenko}. Furthermore, the statement of the sub-critical external field regime is almost identical to that of the zero external field case (see Theorem \ref{thm:ALR}). Lastly, one can verify that the deterministic shift applied to the free energy in statement 3 is $O(N^{-1})$ from the shift used in the other two regimes, using the fact that $q$ as above is $q \approx \frac{h^2}{1-\b^2}$ as $h\downarrow 0$. 

The methods of smart-path interpolation and characteristic trick, of quite similar flavor to the works \cites{Chatterjee09, Chatterjee2014book, Chatterjee19} of Chatterjee and \cite{ChenLam19} of Chen and Lam, are used to prove the three statements. As we have seen in our discussion of these previous works, the methods rely crucially on the assumption of Gaussian disorder. Notably, an alternative proof for the sub-critical regime for Gaussian disorder
is also provided in \cite{DeyWu23}, showing that the free energy in this regime is close to free energy in the case of no external field. 

\begin{remark}\label{rmk:nongaussian}
    A large part of \cite{DeyWu23} proves the above central limit theorems in the case where the disorder is not necessarily Gaussian. The authors achieve this goal in the sub-critical and critical regimes, using the approach of cluster expansion. In particular, statements 1 and 2 of Theorem \ref{thm:deywu23_SK} hold for non-Gaussian disorder, provided the interaction coefficients $J_{ij}$ are i.i.d.\@ symmetric with unit variance and finite sixth moment. In that case, the means and variances of the limiting Gaussian distributions in the two regimes depend on the fourth moment $\bE[J_{12}^4]$. This assumption on the disorder is similar to the one used in \cite{AizenmanLebowitzRuelle}, in which the authors assume i.i.d.\@ symmetric disorder whose $k$-moment is finite for all $k$, and prove the central limit theorem also via the cluster expansion approach.
\end{remark}

\subsection{Results for SK model with ferromagnetic Curie--Weiss interaction}

Free energy fluctuations have also been studied in the setting where the SK model has the Curie--Weiss interaction term. As introduced in \eqref{eq:CW}, the Hamiltonian in this case is given by
    \[
    H_N(\s) = \frac{1}{2\sqrt{N}}\sum_{i,j=1}^N g_{ij}\s_i \s_j + \frac{J}{2N}\sum_{i,j=1}^N \s_i\s_j
    \]
where $(g_{ij})_{1\leq i, j \leq N}$ are i.i.d.\@ standard Gaussian random variables and $J$ is a positive constant. 

In \cite{Banerjee2020}, Banerjee proves a limit theorem for fluctuations of the free energy in the \textit{paramagnetic} regime ($\b<1$ and $\b J <1$), in which the system is at high temperature and has weak ferromagnetic strength. In this regime, the free energy is asymptotically Gaussian, with fluctuations of order $N^{-1}$. We note that \cite{Banerjee2020} uses a slightly different convention for the Hamiltonian, and the following theorem has been adjusted to follow our definition. 

\begin{theorem}[SK model with CW interaction \protect{\cite[Theorem 2.1]{Banerjee2020}}]
        For $\b<1$ and $\b J <1$, we have the convergence in distribution as $N \to \infty$:
        \[
        N(F_N(\b) -F(\b)) \to \cN(f_1,\a_1)
        \]
        where $F(\b)=\frac14\b^2$ and
        \begin{align*}
            \a_1 &= -\frac14\b^2 - \frac12 \log(1-\b^2),\\
            f_1 &= -\frac12\log(1-\b J) + \frac14\log(1-\b^2).
        \end{align*}
    \end{theorem}
The proof relies on analyzing the so-called signed cycles. In contrast to cycles arising from cluster expansion as in \cites{AizenmanLebowitzRuelle, DeyWu23}, the cycles in \cite{Banerjee2020} arise due to a refined second moment method devised by the author. In particular, the partition function, rescaled by its expectation, is formulated as a Radon-Nikodym derivative of two probability measures, depending on $N$. In the language of statistics, they rewrite the free energy as a log-likelihood ratio of two measures $\mathbb{P}_N,\mathbb{Q}_N,$ that are induced by the random matrices with entries $\frac{g_{ij}}{2\sqrt{N}}$ and $\frac{g_{ij}}{2\sqrt{N}}+\frac{J}{2N}$ respectively. The Radon-Nikodym derivative (likelihood ratio) is then approximated by the (random) signed cycles, which are shown to be asymptotically Gaussian under the two measures since the cycles closely approximate the fluctuations of linear spectral statistics of a GOE matrix (with zero diagonal entries). This approach is used in a subsequent work to derive the free energy fluctuations of the mixed $p$-spin SK model (with non-zero 2-spin term) at sufficiently high temperature \cite{Banerjee2021}.

\section{Free energy fluctuations in the SSK model}\label{sec:SSK}

We saw in the previous section that the fluctuations of the SK model are much better understood at high temperature than at low temperature.  In contrast, the fluctuations of the SSK are well understood at all temperatures.  One sees Gaussian fluctuations at high temperature (as in the SK model). At low temperature, the fluctuations follow the GOE Tracy--Widom distribution.\footnote{The GOE Tracy--Widom distribution arises in a variety of settings, notably as the rescaled fluctuations of the largest eigenvalue of a GOE matrix.} Finally, the free energy fluctuations at critical temperature converge to a sum that interpolates between the two distributions.  

The successful analysis of this model is due to the fact that many questions about SSK can be directly translated into questions about the eigenvalues of the underlying random matrix.  Drawing on the extensive random matrix literature, this unlocks many additional tools for analyzing SSK.

\paragraph{Contour integral representation of the partition function}
The starting point for much of the recent progress on SSK fluctuations is a contour integral representation for the partition function $Z_N$.  This representation first appeared in the paper of Kosterlitz, Thouless, and Jones \cite{kosterlitz1976spherical}.  More recently, Baik and Lee \cite{BL16} demonstrate that it can be used to prove precise asymptotic formulas for the distribution of the SSK free energy in both the high and low temperature regimes.  Since then, the contour integral representation has been used in numerous papers analyzing the SSK model and its variants \cite{BL16,BL17,BLW18,BL20,Baiketal21,CW21,CWL_BSSK,JKOP2,landon2020fluctuations,LS22,Landon22,nguyen2018central}.

Below, we present the contour integral representation of the SSK partition function along with its short proof, reproduced from \cite{BL16}.  Similar representations can be derived for the SSK with external field, SSK with Curie-Weiss interaction, and bipartite SSK.  The contour integral approach can be used to study overlaps as well as the free energy.

\begin{lemma}[Contour integral for $Z_N$]\label{lem:contourintegral} Given a symmetric matrix $M$ with eigenvalues $\la_1\geq\cdots\geq\la_N$, the partition function $Z_N$ with the corresponding Hamiltonian $H(\s) =\frac12\s^TM\s$ can be expressed as
\beq
Z_N=C_N\int_{\g-\ii\infty}^{\g+\ii\infty}e^{\frac{N}{2}\cG(z)}\dd z
\eeq
where the contour is a vertical line crossing the real axis at any $\g>\la_1$ and
\beq
\cG(z)=\b z-\frac1N\sum_{i=1}^N\log(z-\la_i),\qquad C_N=\frac{\Gamma(N/2)}{2\pi\ii(N\b/2)^{N/2-1}}.
\eeq
\end{lemma}
\begin{proof}
Let $\Lambda=\diag(\la_1,\cdots\la_N)$ and $O$ be an orthogonal matrix satisfying $M=O\Lambda O^T$.  Then, using the change of variable $x=\frac{1}{\sqrt{N}}\sigma$, we have
\beq
Z_N=\frac{1}{|S^{N-1}(1)|}\int_{S^{N-1}(1)}\exp\left(\frac{\b N}{2}\sum_{i=1}^N\la_i x_i^2\right)\dd\Omega_{N-1}(x)
\eeq
where $S^{N-1}(1)$ denotes the sphere in $\RR^N$ of radius 1 and $\Omega_{N-1}$ denotes its surface area measure.  We can express this as
\beq
Z_N=\frac{I(\frac{\b N}{2})}{|S^{N-1}(1)|},\qquad\text{where }I(t)=\int_{S^{N-1}(1)}\exp\left(t\sum_{i=1}^N\la_i x_i^2\right)\dd\Omega_{N-1}(x).
\eeq
Now we consider the related quantity $J(t)=t^{\frac{N}{2}-1}I(t)$.  The key observation is that the Laplace transform of $J(t)$ is a Gaussian integral that can be evaluated in closed form.  By performing this integration and then taking the inverse Laplace transform, one obtains a simplified expression for $J(t)$ and thus $Z_N$.  More specifically, the Laplace transform of $J(t)$ is
\beq
L(z)=\int_0^\infty e^{-zt}J(t)\dd t=\int_0^\infty t^{\frac{N}{2}-1}\int_{S^{N-1}(1)}\exp\left(-t\sum_{i=1}^N(z-\la_i)x_i^2\right)\dd\Omega_{N-1}(x)\dd t
\eeq
Using the change of variable $t=r^2$ and $y=rx$ where $t,r\geq 0$, we convert from spherical to Cartesian coordinates and obtain
\beq
L(z)=2\int_{\RR^N}\exp\left(-\sum_{i=1}^N(z-\la_i)y_i^2\right)\dd^Ny=2\prod_{i=1}^N\sqrt{\frac{\pi}{z-\la_i}}.
\eeq
Taking the inverse Laplace transform of this expression and requiring that $\mathrm{Re} (z) >\la_1$, we arrive at a contour integral representation for $J(t)$.  Evaluating at $t=\frac{N\b}{2}$, one can obtain the contour integral representation for $Z_N$.
\end{proof}

\subsection{Results for SSK model without external field}
Baik and Lee prove that the SSK has Gaussian fluctuations of order $N^{-1}$ at high temperature, but GOE Tracy--Widom fluctuations of order $N^{-2/3}$ at low temperature \cite{BL16}.  A more precise statement of their results is given below.  While this survey focuses on the SSK model with a GOE interaction matrix, Baik and Lee actually prove their result for a more general class of interaction matrices (see \cite{BL16} for further details).  Note that their paper uses a slightly different scaling of $\b$, but we have written the theorems to be consistent with our notation.

\begin{theorem}[SSK at high and low temperature \cite{BL16}] 
When no external field is present ($h=0$), the free energy of SSK exhibits the following convergence in distribution as $N\to\infty$.
\begin{enumerate}[(i)]
\item In the high temperature regime, $\b<1$,
\beq
N(F_N(\b)-F(\b))\to\cN(-\frac12 \s^2,\s^2)
\eeq
where
$\s^2=-\frac12\log(1-\b^2)$.  

\item In the low temperature regime, $\b>1$,
\beq
\frac{2N^{2/3}}{\b-1}(F_N(\b)-F(\b))\to TW_1.
\eeq
\end{enumerate} 
The limiting free energy $F(\b)$ in this setting is given by the formula
\beq\label{eq:limitingF_SSK}
F(\b)=\begin{cases}
\frac14\b^2& \text{for }\b\leq1,\\
\b-\frac12\log\b-\frac34 & \text{for }\b\geq1.
\end{cases}
\eeq
\end{theorem}

We note that the formula for $\s^2$ is slightly different if one takes the disorder matrix to have zero on the diagonal or to have non-Gaussian entries.

\paragraph{Rough heuristic} The proof method for this theorem uses Lemma \ref{lem:contourintegral} along with a delicate steepest descent analysis.  
\begin{itemize}
\item At high temperature, the saddle point $\gamma$ of the function $\cG$ in Lemma \ref{lem:contourintegral} is well approximated by a deterministic quantity, away from the random spectrum.  Thus, $\cG(\gamma)$ has Gaussian fluctuations given by the CLT for linear statistics of random matrices \cite{Johansson98}.
\item At low temperature, $\gamma$ is extremely close to the largest eigenvalue $\la_1$, so the behavior of the free energy is controlled by the fluctuations of $\la_1$, which follow the GOE Tracy--Widom distribution.
\end{itemize}
\paragraph{Critical temperature results} The dramatic contrast between the high- and low-temperature regimes in SSK lead to an obvious question of what happens in the transition between the two.  This is answered in two papers, which appear around the same time, by Landon \cite{Landon22} and by Johnstone, Klochkov, Onatski, and Pavlyshyn \cite{JKOP2}.  The theorem below is the formulation that appears in \cite{JKOP2}.  We present it in the setting of SSK with GOE interaction, although the paper generalizes to Wigner interaction matrices (with suitable moment conditions).  In addition, the paper includes a treatment of the case with a ferromagnetic interaction (see Section \ref{subsec:SSKferro}).

\begin{theorem}[SSK at critical temperature \cite{JKOP2}]\label{thm:SSK_crit}
When no external field is present, and the inverse temperature is $\b=1+bN^{-1/3}\sqrt{\log N}$ for some fixed $b\in\RR$, the free energy of SSK exhibits the following convergence in distribution as $N\to\infty$.
\beq
\frac{N}{\sqrt{\frac16\log N}}\left(F_N(\b)-F(\b)+\frac{\log N}{12N}\right)\to\cN(0,1)+\sqrt{\frac32}\max\{0,b\}\TW_1
\eeq
\end{theorem}

The paper \cite{Landon22} also analyzes the free energy for $\b=1+bN^{-1/3}\sqrt{\log N}$, but with a slightly different focus than that of \cite{JKOP2}.  In addition to the case of fixed $b$, \cite{Landon22} considers $b\to\infty$ and $b\to-\infty$ at rates slower than $N^{1/3}/\sqrt{\log N}$ and shows that the fluctuations in these two cases are Tracy-Widom and Gaussian, respectively.  On the other hand, for the case of fixed $b>0$, \cite{Landon22} proves only tightness rather than the interpolating limit.

As with the results for the high- and low-temperature regimes, Theorem \ref{thm:SSK_crit} relies heavily on random matrix theory.  In particular, one needs to understand the distribution of the sum appearing in the function $G$ in Lemma \ref{lem:contourintegral}, as $z$ approaches the spectral edge.  In this case, one can no longer use the general CLT for linear statistics of random matrix eigenvalues, but rather an ``edge CLT'' is needed (see \cite{JKOP1,LambertPaquette21,AugeriButezZeitouni}).

\subsection{Results for SSK model with external field}

As with the SK model, the free energy fluctuations of SSK with an external field are Gaussian of order $N^{-1/2}$.  Chen, Dey, and Panchenko \cite{ChenDeyPanchenko} prove this in the SK context as part of a more general result for mixed $p$-spin models with even $p$ (see Section \ref{sec:SK+ext} for more a more detailed description of \cite{ChenDeyPanchenko}).  They indicate that a similar approach should extend the result to the spherical case.  This finding was confirmed via a contour integral method \cite{Baiketal21,landon2020fluctuations}.  In the zero-temperature setting, Chen and Sen derive the ground-state energy fluctuations for the general mixed-$p$ spherical model, which includes SSK \cite{ChenSen}. 

\begin{theorem}[SSK with external field \cite{ChenDeyPanchenko,Baiketal21,landon2020fluctuations}]
For fixed external field strength $h>0$, the free energy of SSK at any inverse temperature $\b>0$ exhibits the following convergence in distribution as $N\to\infty$:
\beq
N^{1/2}(F_N(\b,h)-F(\b,h))\to \cN(0,v_{h,\b})
\eeq
where formulas for $F(\b,h)$ and $v_{h,\b}$ can be found in Results 5.5-5.6 of \cite{Baiketal21}. Note in particular
\begin{itemize}
\item As $h\to0$, the value of $F(\b,h)$ converges to the formula given in \eqref{eq:limitingF_SSK};
\item The value of $v_{h,\b}$ changes slightly depending on whether one chooses the external field to be random (as in \cite{Baiketal21}) or deterministic (as in \cite{ChenDeyPanchenko,landon2020fluctuations}).  See Section 5.3 of \cite{Baiketal21} for a comparison.
\end{itemize}
\end{theorem}

\paragraph{External field transition}
Summarizing the above results, we see that the free energy fluctuations for SSK are:
\begin{itemize}
\item Gaussian $N^{-1}$ at high temp with no external field ($\b<1,h=0$),
\item Tracy--Widom $N^{-2/3}$ at low temp with no external field ($b>1,h=0$),
\item Gaussian $N^{-1/2}$ at any temp with external field ($h>0$).
\end{itemize}
Given this picture, it is natural to ask about the fluctuations as $h\downarrow0$, for example if one takes $h=N^{-\a}$ for some $\a>0$.  This question was first investigated in the zero-temperature limit, where Fyodorov and Le Doussal \cite{FyodorovleDoussal} demonstrate, using non-rigorous saddle point analysis, that there is a transition in the complexity (number of critical points of the Hamiltonian) when $h\sim N^{-1/6}$.  They also compute large deviations for the ground state energy (see \cite{DemboZeitouni} for rigorous version).  Kivimae \cite{kivimae2019critical} follows up on the conjecture of \cite{FyodorovleDoussal} and obtained the limiting fluctuations of the ground state energy for $h\sim N^{-1/6}$.

Moving beyond the zero temperature setting, two simultaneous papers by Baik, Collins-Woodfin, Le Doussal, and Wu \cite{Baiketal21} and by Landon and Sosoe \cite{landon2020fluctuations} obtained the free energy fluctuations of SSK for transitional scalings of the external field in both the high- and low-temperature regimes.  At high-temperature, the transitional scaling is $h\sim N^{-1/4}$ and the fluctuations are Gaussian, thus matching the result for SK at high temperature \cite{DeyWu23}. 
 In the low-temperature regime, the transitional scaling is $h\sim N^{-1/6}$, similar to the result found by \cite{FyodorovleDoussal,kivimae2019critical} at zero temperature.  Furthermore, the fluctuations can be expressed as a function of the GOE Airy point process, and this function interpolates between the Gaussian distribution that emerges for fixed $h>0$ and the Tracy-Widom distribution that occurs when $h=0$.  The techniques in these papers build on those of \cite{BL16}, using the contour integral from Lemma \ref{lem:contourintegral} along with spectral properties of random matrices.

\begin{theorem}[SSK with critically scaled external field \cite{Baiketal21,landon2020fluctuations}] 
When the external field is critically scaled as described below, the free energy of SSK exhibits the following convergence in distribution as $N\to\infty$.
\begin{enumerate}[(i)]
\item If $\b<1$ and $h=HN^{-1/4}$ for fixed $H>0$, then
\beq
N(F_N(\b,h)-F(\b,h))\to\cN(-\frac12 \s^2,\s^2)
\eeq
where $\s^2=-\frac12\log(1-\b^2)+\frac{H^4\b^4}{2(1-\b^2)}$.

\item If $\b>1$ and $h=HN^{-1/6}$ for fixed $H>0$, then
\beq
\frac{2N^{2/3}}{\b-1}(F_N(\b,h)-F(\b,h))\to \tilde{\cF}_{\b,H}
\eeq
where $\tilde{\cF}_{\b,H}$ is a function of Gaussian random variables and a GOE Airy point process.
\end{enumerate} 
The limiting free energy $F(\b,h)$ in this setting is given by the formula
\beq
F(\b,h)=\begin{cases}
\frac14\b^2+\frac12\b^2h^2& \text{for }\b\leq1,\\
\b-\frac12\log\b-\frac34+\frac12\b h^2 & \text{for }\b\geq1.
\end{cases}
\eeq
\end{theorem}
The random variable $\tilde{\cF}_{\b,H}$ is of particular interest from the perspective of random matrix theory, because it shows that, in this transitional regime, the fluctuations are influenced by both the eigenvalues and eigenvectors of the disorder matrix. More specifically, $\tilde{\cF}_{\b,H}$ can be expressed as a sum of Gaussian random variables and a GOE Airy point process, where the Gaussian contribution comes from overlaps of the eigenvectors with the external field vector, while the GOE Airy point process comes from the eigenvalues.  In contrast, when $h=0$, the fluctuations of the free energy at low temperature are Tracy-Widom, dominated by the contribution of the largest eigenvalue (the other eigenvalues and eigenvectors have negligible contribution).  On the other side, when $h>0$ is of constant order, the fluctuations are Gaussian, coming from the overlaps of the eigenvectors with the external field vector (and the eigenvalues have negligible contribution).

\paragraph{Non-linear external field} One possible generalization of the external field Hamiltonian is to replace the inner product $\sigma^Tv$ with some (possibly non-linear) function of it, i.e. $f(\sigma^Tv)$.  Belius and Fr\"{o}ber \cite{beliusfrober} consider the fluctuations of the ground state energy in such a model.  They prove that the fluctuations are Gaussian of order $N^{-1/2}$ with explicit mean, variance, and next-order corrections, under the assumption that $f\in C^3([-1,1])$ and that the function $\mathfrak{B}(\a):=f(\a)+\b\sqrt{2(1-\a^2)}$ has a unique global maximizer $\hat{\a}\neq0$ with $\mathfrak{B}''(\hat{\a})<0$.

\subsection{Results for SSK model with ferromagnetic Curie-Weiss interaction}\label{subsec:SSKferro}

The fluctuations of the free energy in this model are studied by Baik, Lee, and Wu \cite{BL17,BLW18} and Johnstone, Klochkov, Onatski and Pavlyshyn \cite{JKOP2} and are now fully understood in the three phases as well as the transitions between them.  We present here the version for a GOE disorder matrix, although the original papers contain some generalizations to other Wigner matrices.

\begin{theorem}[SSK with CW interaction \cite{BL17}] In the three phases, $F_N$ exhibits the following convergence in distribution as $N\to\infty$. 
\begin{enumerate}[(i)]
\item Spin glass: If $\b>1$ and $J<1$, then
\beq
\frac{2N^{2/3}}{\b-1}(F_N(\b)-F(\b))\to TW_1.
\eeq
\item Paramagnetic: If $\b<1$ and $\b<1/J$, then
\beq
N(F_N(\b)-F(\b))\to\cN(f_1,\a_1)
\eeq
where $\a_1=-\frac12\log(1-\b^2)$ and $f_1=-\frac12(\a_1+\log(1-\b J))$.
\item Ferromagnetic: If $J>1$ and $\b>1/J$, then
\beq
N^{1/2}(F_N(\b)-F(\b))\to\cN(0,\a_2)
\eeq
where $\a_2=\frac12(1-J^{-2})(4\b-J^{-1})^2.$
\end{enumerate}
Here the limiting free energy $F(\b)$ is given by the formula
\beq\label{eq:F(beta)SSK+CW} 
F(\b)=\begin{cases} \b-\frac12\log\b-\frac34&\text{for spin glass,}\\
\frac14\b^2&\text{for paramagnetic,}\\
\frac\b2(J+\frac1J)-\frac12\log(\b J)-\frac{1}{4J^2}-\frac12&\text{for ferromagnetic.}
\end{cases}
\eeq
\end{theorem}

\begin{theorem}[SSK with CW interaction at phase transitions \cite{BL17,BLW18,JKOP2}] At the boundaries between each pair of phases and at the triple point, $F_N$ exhibits the following convergence in distribution as $N\to\infty$.
\begin{enumerate}[(i)]
\item Spin glass $\leftrightarrow$ Ferromagnetic \cite{BL17}: If $\b>1$ and $J=1-wN^{-1/3}$ for fixed $w\in\RR$, then
\beq
\frac{2N^{2/3}}{\b-1}(F_N(\b)-F(\b))\to \TW_{1,w}
\eeq
where $F(\b)$ is defined as in \eqref{eq:F(beta)SSK+CW} for the spin glass regime and  $\TW_{1,w}$ is a parametrized family characterized in Theorems 1.5 and 1.7 of \cite{Bloemendal_2012}, interpolating between Tracy-Widom and Gaussian.
\item Ferromagnetic $\leftrightarrow$ Paramagnetic \cite{BLW18}: If $J>1$ and $\b=J^{-1}+BN^{-1/2}$ for fixed $B\in\RR$, then
\beq
N\left(F_N(\b)-F(\b)-\frac{\log N}{4N}\right)\to \cG_1+Q(\cG_2)
\eeq
where $F(\b)$ is defined as in \eqref{eq:F(beta)SSK+CW} for the ferromagnetic regime, $(\cG_1,\cG_2)$ has bivariate Gaussian distribution with explicit mean and variance depending on $J$, and $Q$ is an explicit function depending on both $J$ and $B$ (in the GOE case that we consider $\cG_1,\cG_2$ are independent, but the theorem is more general).
\item Spin glass $\leftrightarrow$ Paramagnetic \cite{JKOP2}: If $J<1$ and $\b=1+bN^{-1/3}\sqrt{\log N}$ for fixed $b\in\RR$, then
\beq
\frac{N}{\sqrt{\frac16\log N}}\left(F_N(\b)-F(\b)+\frac{\log N}{12N}\right)\to\cN(0,1)+\sqrt{\frac32}b_+\TW_1
\eeq
where $F(\b)$ is defined as in \eqref{eq:F(beta)SSK+CW} for the spin glass ($b\geq0$) and paramagnetic ($b\leq0$) regimes, $b_+=\max\{b,0\}$, and $\TW_1$ is independent of $\cN(0,1)$.
\item Triple point \cite{JKOP2}: If $\b=1+bN^{-1/3}\sqrt{\log N}$ and $J=1-wN^{-1/3}$ for fixed $b,w\in\RR$, then
\beq
\frac{N}{\sqrt{\frac16\log N}}\left(F_N(\b)-F(\b)+\frac{\log N}{12N}\right)\to\cN(0,1)+\sqrt{\frac32}b_+\TW_{1,w}.
\eeq
\end{enumerate}
\end{theorem}

As we saw in the other SSK fluctuation results, the techniques used to prove these theorems are rooted in a contour integral representation of the partition function and random matrix theory.  However, the model with CW interactions reveals new and interesting connections to random matrix theory.

In the transition between the spin glass and ferromagnetic regimes (which have Tracy-Widom and Gaussian fluctuations respectively), one sees a parametrized family of distributions that corresponds to those studied by Bloemendal and Vir\'ag \cite{Bloemendal_2012} in the context of spiked random matrices.  In their context,  $\TW_{1,w}$ corresponds to the limiting distribution of $N^{2/3}(2-\la_1(w))$ where $\la_1(w)$ is the largest eigenvalue of a spiked GOE matrix with spike strength $1-wN^{-1/3}$ (where a spike strength of 1 is critical and referred to as the BBP transition).  Note that $\TW_{1.w}$ can be viewed as a perturbation of the Tracy-Widom GOE distribution, where $TW_1$ corresponds to the large-$w$ limit of $\TW_{1,w}$.

\section{Free energy fluctuations in multi-species models}\label{sec:multispecies}

The SK and SSK models discussed in Sections \ref{sec:SK} and \ref{sec:SSK} are both homogeneous models, where the particles are indistinguishable. A natural question is how adding inhomogeneity, breaking the symmetry among the particles, quantitatively affects the behavior of the model. This motivates the \textit{multi-species} models, in which the particles are divided into a fixed number of species, and the proportion of each species is asymptotically fixed as the number of spins goes to infinity. Moreover, the variance of the spin interactions $g_{ij}$ is no longer constant, but depends on the species structure.

\begin{definition}[Multi-species 2-spin model]
    Fix $k\geq 1$. Let $\Lambda = \diag(\a_1,\dots, \a_k)$ be a matrix satisfying $\a_i\in (0,1)$ for all $i$ and $\sum_{i=1}^k \a_i=1$. Let $\Delta^2$ be a $k\times k$ symmetric matrix of non-negative entries. 
    A multi-species model with species density matrix $\Lambda$ and variance profile matrix $\Delta^2$ satisfies the following assumptions.
    \begin{itemize}
        \item Given $N \geq k$, there is a partition of spins $\{1,2,\dots, N\} = \cup_{s=1}^k I_s$, where the species density
        \[
        \Lambda_N:=\frac1N\diag(|I_1|,...,|I_k|) \quad \text{satisfies} \quad \lim_{N\to\infty}\Lambda_N=\Lambda.
        \]
        \item Let $(g_{ij})_{1 \leq i, j\leq N}$ be independent centered Gaussian variables such that $\bE[g_{ij}^2] = \Delta^2_{s,t}$ for $i\in I_s$ and $j\in I_t$. The Hamiltonian of the model, at each spin configuration $\s\in \Sigma_N$, is defined as
        \beqq
        H_N(\s) = \frac1{\sqrt N}\sum_{i,j=1}^N g_{ij}\s_i\s_j.
        \eeqq
    \end{itemize}
    Here, $\Sigma_N = \{\pm 1\}^N$ in the case of Ising spins and is the product of spheres $S_{|I_1|-1}  \times  \dots \times S_{|I_k|-1}$ in the case of spherical model. 
\end{definition}

We discuss the literature on the multi-species extensions of the SK model in Section \ref{subsec:msk}, where we refer to them via the shorthand MSK. The discussion includes recent fluctuation results under weak external field of \cite{DeyWu23}, and the variance bound of the free energy that we obtain for the MSK model at the critical temperature \cite{CWL_variancebound}, which is analogous to the bound obtained in \cite{ChenLam19} for the SK model. 

The literature on multi-species spherical SK models is discussed in Section \ref{subsec:msk_spherical}, with a strong focus on the particular example that is the bipartite spherical SK model. 

\subsection{Results for SK multi-species model}\label{subsec:msk}

The inhomogeneity in the MSK model leads to several mathematical challenges. In the case where $\Delta^2$ is positive-definite, Barra, Contucci, Mingione and Tantari propose a Parisi formula for the free energy and prove the upper bound via Gaussian interpolation in \cite{Barra15}. The matching lower bound, which does not rely on the positive-definite assumption, is verified by Panchenko \cite{Panchenko15}, thus providing the limiting free energy for all fixed $\b>0$. In contrast, limiting free energy in the case of indefinite $\Delta^2$ remains open. 

There has also been progress towards free energy fluctuations, under various assumptions such as definiteness of $\Delta^2$, sufficiently high temperatures, and strength of the external field. As observed from Section \ref{sec:beyondscope}, the discussion of different regimes for the free energy involve the notions of RS and RSB phases in a nontrivial way. This is also the case for multi-species models. 

In the MSK model with $k$ species, each species $s \in \{1, \dots, k\}$ has an order parameter $\mu_s$, which is the limiting distribution of the species-specific overlap $|R_{1,2}^s|$, where 
\[
R_{1,2}^s :=\frac{1}{|I_s|}\sum_{i \in I_s} \s^1_s\s^2_s.
\]
We say that the system is replica symmetric if $\mu_s$ is a singleton for every $s$ (i.e. the overlaps all concentrate). Otherwise, we say the system has replica symmetry breaking. This extends the existing notions of RS and RSB for the SK model. Note that classification of the RSB phase becomes more complicated for multi-species models, as it is unclear how much symmetry breaking in one species affects symmetry breaking in others. 

The work \cite{BatesSlomanSohn19} of Bates, Sloman and Sohn considers the two-species SK model. Under the assumption that 
$\Delta^2 := \begin{pmatrix}  \Delta^2_{11} & \Delta^2_{12}\\
\Delta^2_{12} & \Delta^2_{22}
\end{pmatrix}$ is positive-definite and $h>0$, the authors show that the model has RSB when $\b >\b_c(\g)$, where 
\beq\label{eq:bcrit_BSS}
\b_c(\g):= \left(\la_1  \Delta^2_{11} + \g_2 \Delta^2_{22}  + \sqrt{(\g_1 \Delta^2_{11} -\g_2  \Delta^2_{22}) + 4\g_1\g_2 \Delta^4_{12} } \right)^{-1/2}
\eeq
and $\g = (\g_1,\g_2)^T$ is explicitly described. This yields a two-species analogue of the RSB condition provided by Toninelli \cite{Toninelli_2002} for SK model. Furthermore, from \cite{BatesSlomanSohn19}, it is expected that the system is RS for sufficiently high temperature $\b<\b_c(\la)$, where $\b_c(\la)$ is given by \eqref{eq:bcrit_BSS} with $\g$ replaced by the species densities $\la := (\la_1, \la_2)^T$. We observe that $\b_c(\la)< \b_c(\g)$ and $\b_c(\la)$ matches the critical value $\b_c=1$ in the case of the SK model. 

The results of \cite{BatesSlomanSohn19} are extended to $k$-species SK, for $k\geq 2$, in the work \cite{DeyWu21} of Dey and Wu. In particular, the multi-species version of $\b_c(\la)$ is of the form 
\beq\label{eq:bcrit_msk}
\b_c= \rho(\Delta^2\Lambda)^{-1/2},
\eeq
where $\rho(\cdot)$ denotes the spectral radius. First, the authors of \cite{DeyWu21} derive a law of large numbers and a central limit theorem for the free energy at every $\b< \b_0$, for some $\b_0<\b_c$. A similar result for the two-species case is obtained independently, around the same time, by Liu \cite{Liu21}.

\begin{theorem}[MSK model at high temperature, \protect{\cite[Theorem 1.9]{DeyWu21}}]
Let $\b_c$ be as in \eqref{eq:bcrit_msk}. Define $\b_0$ by 
\[
\b_0 = \b_c/\sqrt{4(1+\mathbbm{1}\{\Delta^2 \text{ is indefinite}\})}.
\]
If $\b < \b_0$ and $h>0$, then 
\[
N^{1/2}\left(F_N(\b) - \mathrm{RS}(\b, h) \right) \to \cN(0, b(\b,h))
\]
where $\mathrm{RS}(\b, h)$ denotes the multi-species replica-symmetric solution. 
\end{theorem}

\begin{remark}
The theorem is stated for the case of positive external field, although the authors of \cite{DeyWu21} provide a sketch for the case of no external field where the moment method of \cite[Chapter 11]{Talagrandvol2} can be applied. We refer to \protect{\cite[Theorem 1.9]{DeyWu21}} for the full expressions of $\mathrm{RS}(\b, h)$ and $b(\b,h)$. 
\end{remark}

In a subsequent work \cite{DeyWu23}, Dey and Wu provide fluctuation results for MSK model in the case of high temperature and weak external field. These provide the multi-species extension of their Theorem \ref{thm:deywu23_SK} for the SK model, and hold for general variance profile matrix $\Delta^2$. The theorem statement we provide here omits the descriptions of the asymptotic means $F(\b), \widehat{F}(\b)$ and the limiting means and variance, which are the multi-species analogues of the ones in the SK version, Theorem \ref{thm:deywu23_SK}. In addition, Remark \ref{rmk:nongaussian} on Theorem \ref{thm:deywu23_SK} also applies here, commenting that Theorem \ref{thm:msk_weakext} as provided in \cite{DeyWu23} holds in more generality.

\begin{theorem}[MSK model at high temperature, under weak external field  \protect{\cite[Theorem 1.8]{DeyWu23}}]\label{thm:msk_weakext}
Consider a MSK model with Gaussian disorder. Assume that the external field is scaled as $h = \rho N^{-a}$ for some positive constants $\rho$ and $\a$. We have the following statements of convergence in distribution as $N \to \infty$.
\begin{enumerate}
     \item (Sub-critical) For $\a>1/4$ and $\b<\b_c$,
        \[
        N \left(F_N - F(\b) \right) \to \cN(-\frac12 V_1^2, V_1^2).
        \]
        \item (Critical) For $\a=1/4$ and $\b<\b_c$,
        \[
        N \left(F_N - F(\b) \right) \to \cN(-\frac12 \left(V_1^2 + \rho^4 V_2^2 \right), V_1^2+\rho^4 V_2^2).
        \]
        \item (Super-critical) There exists $\b_0\leq \b_c$ such that, for $\a<1/4$ and $\b<\b_0$, then 
        \[
        \frac{\sqrt{N}}{h^2}\left(F_N - \widehat{F}(\b)\right) \to \cN(0, V_2^2).
        \]
\end{enumerate}
\end{theorem}
Although the expression of $\widehat{F}(\b)$ for the sub-critical regime involves the solution $\boldsymbol{q}$ to the fundamental equations and is more complicated, it is of the same form as $F(\b)$ after applying the asymptotics of $\boldsymbol{q}$. 

Lastly, we consider the critical temperature regime. We extend Theorem \ref{thm:chenlam19} of Chen and Lam for the SK model to hold for a general MSK model, provided that the variance profile $\Delta^2$ is positive semi-definite. At the critical temperature conjectured in \eqref{eq:bcrit_msk}, we show that the order of fluctuations of the MSK free energy under this assumption are at most $N^{-1}\log N$. We also obtain a variance bound for the case where the system approaches the critical temperature from the low temperature side.

\begin{theorem}[MSK at critical temperature, \cite{CWL_variancebound}]\label{thm:main}
    Consider an MSK model with species density matrix $\Lambda=\Lambda_N+O(N^{-1})$ and variance profile matrix $\Delta^2$.Let $\b_c$ be as given in \eqref{eq:bcrit_msk}. If $\Delta^2$ is positive semi-definite, then the following statements hold for the free energy $F_N(\b)$.
    \begin{enumerate}
    \item There exists a constant $C>0$ such that 
    \beq
    	\Var\left(F_N(\b_c)\right) \leq C \left( (\log N)^2 +1\right).
    \eeq
    \item For fixed $\a>0$ and $d>0$, there is a constant $C>0$, depending only on $\a$ and $\d$, such that
    \beq
    	\Var\left(F_N(\sqrt{\b_c^2+dN^{-\a}})\right) \leq C \left( (\log N)^2 + N^{1-\a}\right).
    \eeq
    \end{enumerate}
\end{theorem}
The proof follows rather closely the method of Chen and Lam in \cite{ChenLam19}, except for appropriate modifications to accommodate the multi-species setting and an adaptation of Talagrand's argument used for the characteristic function.

\subsection{Results for SSK multi-species model}\label{subsec:msk_spherical}
When it comes to spherical multi-species models, there has been much work on characterizing the limiting free energy for various multi-species set-ups \cite{AuffingerChenBSSK,Subag23,BatesSohn22,batessohn25balanced}.  However, precise results on the fluctuations of the free energy are only available in the special case of bipartite spherical SK (BSSK).  Baik and Lee \cite{BL20} showed that the free energy of BSSK is similar to that of SSK in the sense that it has Gaussian fluctuations of order $N^{-1}$ at high temperature, but Tracy-Widom fluctuations of order $N^{-2/3}$ at low temperature. In \cite{CWL_BSSK}, the authors of the current survey then show that, in a window around the critical temperature, the fluctuations of BSSK converge to an independent sum of Gaussian and Tracy-Widom random variables, where the precise formula depends on the species-size ratio.

Before presenting the theorems for BSSK, we remark on the striking similarity between SSK and BSSK and try to provide some intuition behind this similarity.  On the one hand, we saw in the previous section that the bipartite model with Ising spins (BSK) poses considerable challenges that do not arise in analyzing SK, particularly due to the indefinite covariance structure.  In the spherical setting, on the other hand, BSSK, like SSK, has a contour integral representation for its partition function and is thus amenable to analysis via techniques from random matrix theory.  

The key idea is that, for the various spherical spin glasses discussed in this survey, the free energy fluctuations are driven by the behavior of eigenvalues in the underlying random matrix model.  For SSK, the relevant matrix is GOE; for SSK with a Curie-Weiss term, this becomes a spiked GOE; and for BSSK, the relevant matrix is the real Wishart or Laguerre Orthogonal ensemble (LOE), constructed as $\frac1m JJ^T$ where $J$ is an $n\times m$ matrix (with $n\leq m)$ of i.i.d. standard Gaussians. 
With this picture in mind, it is not surprising that BSSK behaves similarly to SSK, because their underlying random matrices have similar spectral properties.  In particular, both GOE and LOE display eigenvalue rigidity \cite{ErdosYauYin,PillaiYin}, CLT for linear statistics of eigenvalues \cite{Johansson98,BaiSilverstein04}, and convergence of the largest eigenvalue to a Tracy-Widom distribution \cite{soshnikov1999universality,Soshnikov2002samplecov}.

\begin{theorem}[BSSK at high and low temperature \cite{BL20}]\label{thm:BSSK_highlowtemp} Consider a BSSK model with species sizes $N_1,N_2$ such that $N_1+N_2=N$ and their ratio converges as $N_1/N_2=\la+O(N^{-1-\delta})$ for some $\lambda\in(0,1)$ and some $\d>0$. Then the critical inverse temperature is $\b_c:=\sqrt{1+\la}/\la^{1/4}$ and the free energy exhibits the following convergence in distribution as $N_1,N_2\to\infty$.
\begin{enumerate}[(i)]
\item In the high temperature regime, $\b<\b_c$,
\beq
N(F_{N_1,N_2}(\b)-F(\b))\to\cN(-\frac12\s^2-\log2,\s^2),
\eeq
where $\s^2=-\frac12\log(1-(\b/\b_c)^4),$ with a slightly more complicated formula in the case of a general Wigner matrix rather than GOE.
\item In the low temperature regime, $\b>\b_c$,
\beq
\frac{N^{2/3}}{A(\b,\la)}(F_{N_1,N_2}(\b)-F(\b))\to\TW_1,
\eeq
where the formula for $A(\b,\lambda)$ is given in Theorem 2.3 of \cite{BL20}.
\end{enumerate}
The limiting free energy $F(\b)$ is given in Theorem 2.1 of \cite{BL20} (we omit the formula here but note that it depends on the ratio $\lambda$ 
and is continuous but non-differentiable at $\b_c$).
\end{theorem}

\begin{theorem}[BSSK at critical temperature \cite{CWL_BSSK}]\label{thm:BSSK_crit} Consider a BSSK model where $N_1/N_2=\la+O(N^{-1})$ for some $\lambda\in(0,1)$ and let $\b=\b_c+bN_1^{-1/3}\sqrt{\log N_1}$ for some fixed $b\in\RR$.  Then the free energy exhibits the following convergence in distribution as $N_1,N_2\to\infty$.
\beq
\frac{N}{\sqrt{\frac16\log N_1}}\left(F_{N_1,N_2}(\b)-F(\b)+\frac{\log N_1}{12N}\right)\to\cN(0,1)+C_\la\max\{0,b\}\TW_1
\eeq
where $F(\b)$ and $C_\la$ are given in Theorem 1.1 of \cite{CWL_BSSK}.
\end{theorem}
Both of the above theorems rely on a contour integral representation for the partition function, which is proved in \cite{BaikLeeBipartite}.  This formula involves a double rather than single contour integral and the location of the (random) saddle point of the integrand depends on both variables.  Thus the contour analysis is somewhat more complicated than in the SSK case.  To treat the critical temperature case, an additional random matrix theory input is needed.  As with the SSK critical temperature analysis, Theorem \ref{thm:BSSK_crit} requires a CLT for the sum $\frac1N\sum_{i=1}^N\log(z-\la_i)$ where $\{\la_i\}$ are eigenvalues and $z$ is near the spectral edge.  This edge CLT was proved for Laguerre matrices (the relevant ensemble for BSSK) in \cite{CWL_CLT}.

\paragraph*{Acknowledgments}
The authors wish to thank Jinho Baik, Erik Bates, Wei-Kuo Chen, and Aukosh Jagannath for helpful discussions. 
E. Collins-Woodfin acknowledges the support of Fonds de recherche du Qu\'ebec – Nature et technologies
(FRQNT) postdoctoral training scholarship (DOI https://doi.org/10.69777/344253), Centre de recherches math\'ematiques (CRM) Applied
math postdoctoral fellowship, and Institut Mittag-Leffler (IML) Junior Fellowship (Swedish Research Council, grant no. 2021-06594).
H.\@ G.\@ Le acknowledges the support of the Canada Research Chairs program [CRC-2022-00142] and  the University of Waterloo Mathematics Faculty Research Fund.

\newpage

\bibliographystyle{abbrv}
\bibliography{Masterbibliography_HanLizbee}

\end{document}